\documentclass[proceedings,submission%
]{dmtcs}

\usepackage[latin1]{inputenc}
\usepackage{subfigure}
\usepackage{tikz}
\usepackage{amssymb,amsmath,latexsym,enumerate}
\usepackage{cite}

\author{Matthias Beck\addressmark{1}\thanks{Research partially supported by NSF grant DMS-0810105.}\and Yvonne Kemper\addressmark{2}\thanks{Research partially supported by NSF grant DMS-0914107 and by NSF VIGRE grant DMS-0636297.}}
\title{Flows on Simplicial Complexes}
\address{\addressmark{1}Department of Mathematics, San Francisco State University, San Francisco, CA 94132, USA\\
\addressmark{2}Department of Mathematics, University of California at Davis, Davis, CA 95616, USA}

\newtheorem{theorem}{Theorem} % [section]

\newtheorem{proposition}[theorem]{Proposition}
\newtheorem{lemma}[theorem]{Lemma}

\newenvironment{definition}[1][Definition]{\begin{trivlist}
\item[\hskip \labelsep {\bfseries #1}]}{\end{trivlist}}

\DeclareMathOperator{\lk}{{\rm{lk}}}
\DeclareMathOperator{\supp}{{\rm{supp}}}
\DeclareMathOperator{\ran}{{\rm{rank}}}
\newcommand{\RR}{\mathbb{R}}
\newcommand{\QQ}{\mathbb{Q}}
\newcommand{\ZZ}{\mathbb{Z}}

\newcommand{\FF}{\mathbb{F}}
\newcommand{\I}{\mathcal{I}}

\renewcommand{\P}{\mathcal{P}}

\newcommand\ehr{\operatorname{ehr}}

\newcommand\commentout[1]{}

\begin{document}
\maketitle
\begin{abstract}
\paragraph{Abstract.}
Given a graph $G$, the number of nowhere-zero $\ZZ_q$-flows $\phi_G(q)$ is known to be a polynomial in $q$.  We extend the definition of nowhere-zero $\ZZ_q$-flows to simplicial complexes $\Delta$ of dimension greater than one, and prove the polynomiality of the corresponding function $\phi_{\Delta}(q)$ for certain $q$ and certain subclasses of simplicial complexes.  

\paragraph{R\'esum\'e.}
E\'tant donn\'e une graphe  $G$, on est connu que le nombre de $\ZZ_q$-flots non-nuls $G$ $(q)$ est un polyn\textasciicircum{o}me dans $q$.  Nous \'etendons la d\'efinition de $\ZZ_q$-flots non-nuls pour incluir des complexes simpliciaux de dimension plus grande qu'un, et on montre que le nombre est aussi un polyn\textasciicircum{o}me de la fonction correspondante pour certain valeurs de $q$ et de certaines sous-classes de complexes simpliciaux.

\end{abstract}

%%%%%%%%%%%%%%%%%

\section{Introduction}\label{sec:in}

Flows on graphs were first studied by Tutte (\cite{TutteOrig1,TutteOrig}; see also \cite{Brooks} and \cite{Tutte1976}) in the context of Kirchoff's electrical circuit laws.  Since this time, a great deal of work (see, for instance, \cite{Jaeger-book},\cite{Seymour-book}, \cite[Ch. 6, Sec. 3.B, 3.C]{White-book}, or \cite{MattFlows}) has been done, and applications found in network and information theory, optimization, and other fields.

To define a flow on a graph $G$, we first give an initial orientation to its edges (this orientation is arbitrary, but fixed).  Then, a \emph{$\ZZ_q$-flow} on $G$ is an assignment of values from $\ZZ_q$ to each edge such that modulo $q$, the sum of values entering each node is equal to the sum of values leaving it.  If none of the edges receives zero weight, the $\ZZ_q$-flow is \emph{nowhere-zero}.  In terms of the signed incidence matrix $M$ of $G$, a $\ZZ_q$-flow is an element of the kernel of $M$ mod~$q$.

In this paper, we extend the idea of flows on graphs to flows on simplicial complexes (namely, elements of the kernel of the boundary map mod $q$---we give detailed definitions below), and explore the work done on graph flows in the context of simplicial complexes.
We consider this paper as the starting point of a systematic study of flows on simplicial complexes.
Extending Tutte's polynomiality result for the enumeration of nowhere-zero flows on graphs, our first main result is the following:

\begin{theorem}\label{mainthm}
Let $q$ be a sufficiently large prime number, and let $\Delta$ be a simplicial complex of dimension $d$.  Then the number $\phi_\Delta(q)$ of nowhere-zero $\ZZ_q$-flows on $\Delta$ is a polynomial in $q$ of degree $\beta_d(\Delta)=\dim_{\QQ}(\widetilde{H}_d(\Delta,\QQ))$.
\end{theorem}

In other words, there exists a polynomial $p(x)$ that coincides with $\phi_\Delta(q)$ at sufficiently large primes. In fact, $\phi_\Delta(q)$ coincides with $p(x)$ at many more evaluations:

\begin{theorem}\label{mainthm2}
The number $\phi_\Delta(q)$ of nowhere-zero $\ZZ_q$-flows on $\Delta$ is a quasipolynomial in $q$.
Furthermore, there exists a polynomial $p(x)$ such that $\phi_\Delta(k) = p(k)$ for all integers $k$ that are relatively prime to the period of~$\phi_\Delta(q)$.
\end{theorem}

(Again, we give detailed definitions, e.g., of quasipolynomials and their periods, below.)
For the case where the simplicial complex triangulates a manifold, we prove more.

\begin{proposition}\label{manifoldprop}
Let $\Delta$ be a triangulation of a manifold.  Then
\[\phi_{\Delta}(q) = \begin{cases}
0 & \text{if }\Delta\text{ has boundary,} \\
q-1 & \text{if }\Delta\text{ is without boundary, $\ZZ$-orientable,}\\
0 & \text{if }\Delta\text{ is without boundary, non-$\ZZ$-orientable, and } q \neq 2,\\
1 & \text{if }\Delta\text{ is without boundary, non-$\ZZ$-orientable, and } q=2.
\end{cases}\]
\end{proposition}

In Section \ref{defsection}, we provide background on simplicial complexes and boundary operators, and define flows on simplicial complexes.  We also discuss certain details of the boundary matrix of pure simplicial complexes.  In Section \ref{proofsection}, we prove Theorems \ref{mainthm} and \ref{mainthm2}.  In Section \ref{manifoldsection}, we examine the specific case of triangulated manifolds and prove Proposition \ref{manifoldprop}.  In Section \ref{openquestion}, we present several open problems.

%%%%%%%%%%%%%%%%%

\section{Flows on Simplicial Complexes}\label{defsection}

Recall that an \emph{(abstract) simplicial complex} $\Delta$ on a \emph{vertex set} $V$ is a set of subsets of $V$.  These subsets are called the \emph{faces} of $\Delta$, and we require that
\begin{enumerate}[(1)]
  \item for all $v\in V$, $\{v\}\in \Delta$, and 
  \item for all $F\in\Delta$, if $G\subseteq F$, then $G\in\Delta$.
\end{enumerate}
The \emph{dimension} of a face $F$ is $\dim(F)=|F|-1$, and the \emph{dimension} of $\Delta$ is $\dim(\Delta) = \max\{\dim(F) : \, F\in\Delta\}$.  A simplicial complex is \emph{pure} if all maximal faces (that is, faces that are not properly contained in any other face) have the same cardinality.  In this case, a maximal face is called a \emph{facet}, and a \emph{ridge} is a face of codimension one.  For more background on simplicial complexes, see \cite[Ch.\ 0.3]{Stanley-book} or \cite[Ch.\ 2]{Hatcher-book}.

Let $\Delta$ be a pure simplicial complex of dimension $d$ with vertex set $V=\{v_0,\ldots,v_n\}$.  Assign an ordering to $V$ so that $v_0<v_1<\cdots<v_n$.  Any $r$-dimensional face of $\Delta$ can be written (in terms of this ordering) as $[v_{i_0}\cdots v_{i_r}]$.  Let $\partial$ be the boundary map on the simplicial chains of $\Delta$, defined by
\[\partial [v_{i_0}\cdots v_{i_r}] = \sum_{j=0}^r (-1)^j[v_{i_0}\cdots\widehat{v_{i_j}}\cdots v_{i_r}].\]

Define the boundary matrix $\partial\Delta$ of $\Delta$ to be the matrix whose rows correspond to the ridges and columns to the facets of $\Delta$, and whose entries are $\pm 1$ or $0$, depending on the sign of the ridge in the boundary of the facet.  In the case of a graph, the facets are the edges, and the ridges are the vertices, and, under the natural ordering of the vertices, the boundary matrix is identical to the signed incidence matrix of the graph.  For more background on the boundary map and boundary matrices, see \cite[Ch.~2]{Hatcher-book}.

Now, since a flow on a graph (one-dimensional simplicial complex) is an element of the kernel mod $q$ of the signed incidence (boundary) matrix, we have a natural way to extend the notion of a flow on a graph to a flow on a simplicial complex. 
\begin{definition} A \emph{$\ZZ_q$-flow} on a pure simplicial complex $\Delta$ is an element of the kernel of $\partial\Delta\bmod q$.  A \emph{nowhere-zero $\ZZ_q$-flow} is a $\ZZ_q$-flow with no entries equal to zero $\bmod q$.
\end{definition}

%\begin{example}
As an example, consider the surface of a tetrahedron, with vertices $V=\{1,2,3,4\}$, with the natural ordering.  Its boundary matrix $\partial\Delta$ is:
\[\begin{array}{c|cccc}
%\frac{\mbox{facet}}{\mbox{ridge}} 
 & 124 & 134 & 234 & 123\\
\hline
14 &  -1 &  -1 &  0  & 0\\
24 &  1  &  0  &  -1 & 0\\
34 &  0  &  1  &  1  & 0\\
12 &  1  &  0  &  0  & 1\\
13 &  0  &  1  &  0  & -1\\
23 &  0  &  0  &  1  & 1
\end{array} \]
One instance of a nowhere-zero $\ZZ_q$-flow on $\Delta$ is
$(x_{123},x_{124},x_{134},x_{234})^T = (1,q-1,1,q-1)^T$.
%\end{example}

For a second example, consider a pure simplicial complex $\Delta$ with a vertex $v$ that is contained in every facet, i.e., $\Delta$ is a \emph{cone} over $v$.

\begin{lemma} Let $\Delta$ be a pure simplicial complex that is a cone over a vertex $v$.  Then $\Delta$ has no nontrivial $\ZZ_q$-flows. 
\end{lemma}

\begin{proof}
Order the vertices $v_0, v_1, \dots, v_n$ so that $v= v_n$ is the largest.
Recall that the \emph{link of $v$ in $\Delta$} is defined as
\[
  \lk_{\Delta}(v):= \left\{ F\in\Delta : \, v\not\in F \text{ and } \{v\}\cup F\in\Delta \right\} ,
\]
and the \emph{deletion of $v$ in $\Delta$} is
\[
  \Delta-v := \left\{ F\in\Delta : \, v\not\in F \right\} .
\]
Now consider the boundary matrix $\partial\Delta$ of $\Delta$, with the rows and columns ordered in the following way:
\begin{enumerate}[(a)]
\item Let all ridges containing the vertex $v_n$ come first, and let them be ordered lexicographically.
\item Let all ridges in $\lk_{\Delta}(v_n)$ be next, ordered lexicographically.
\item Order the remaining ridges lexicographically.
\item Let all facets containing $v_n$ come first in the columns, and order them lexicographically.
\item Order the remaining facets (those that do not contain $v_n$) lexicographically.
\end{enumerate}

It is not hard to see that under these conditions the boundary matrix takes the following form:
\vspace{1mm}
\[\begin{array}{|c|c|}
& \\
\partial(\lk_{\Delta}(v_n)) & \mathbf{0}\\
%\mbox{ \tiny{(a)} } & \mbox{ \tiny{(b)} }\\
& \\
\hline
& \\
\mathbf{\pm I} & \\
%\mbox{ \tiny{(c)} } & \partial(\Delta\setminus v_n)\\
& \partial(\Delta\setminus v_n) \\
\cline{1-1}
& \\
\mathbf{0} & \\
%\mbox{ \tiny{(d)} } & \mbox{ \tiny{(e)} }
\end{array} \]
\vspace{1mm}

\noindent
Since $\Delta$ is a cone over $v_n$, the two regions on the right-hand side of the matrix are nonexistent, and $\pm {\bf I}$ extends through all columns.  Therefore, we have a series of rows with precisely one nonzero entry.
\end{proof}

%%%%%%%%%%%%%%%%%

\section{The Number of Nowhere-Zero $\ZZ_q$-Flows on a Simplicial Complex}\label{proofsection}

\subsection{Tutte polynomials of matroids and proof of Theorem \ref{mainthm}}

To study the question of counting nowhere-zero flows, we recall the notion of a \emph{matroid} and refer to, e.g., \cite{Oxley-book}, for relevant background.
\commentout{
\begin{definition} A \emph{matroid} is an ordered pair $M=(E,\I)$, such that $E$ is a finite set and $\I$ is a collection of subsets of $E$ satisfying the following conditions:
\begin{enumerate}[1.]
\item $\emptyset\in\I$
\item if $I\in\I$ and $I'\subseteq I$, then $I'\in\I$
\item if $I_1,I_2\in\I$ and $|I_1|<|I_2|$, then there exists $e\in I_2\setminus I_1$ such that $I_1\cup\{e\}\in\I$.
\end{enumerate}
\end{definition}
The \emph{bases} of a matroid are the maximal independent sets; it follows from the above conditions that they are all of the same cardinality.  An element $e\in E$ that is contained in every base is a \emph{coloop}.  A subset of $E$ is \emph{dependent} if it is not independent, and a minimal dependent subset is called a \emph{circuit}.  The rank of any subset $S\subseteq E$, $r(S)$, is the size of the largest independent set contained in $S$, and the rank of the matroid is $r(M)=r(E)$.  An element $e\in E$ is a \emph{loop} if $r(\{e\})=0$.
}
Associated with every matroid $M=(E,\I)$, where $E$ is the \emph{ground set} and $\I$ the collection of \emph{independent sets} of $M$, is the \emph{Tutte polynomial}
\[T_M(x,y)=\sum_{S\subseteq E}(x-1)^{r(M)-r(S)}(y-1)^{|S|-r(S)} ,\]
where $r(S)$ denotes the \emph{rank} of $S$.
The Tutte polynomial satisfies the following recursive relation:

\[T_M(x,y) = \begin{cases}
T_{M/e}(x,y)+T_{M-e}(x,y) & \text{if $e$ is neither a loop nor a coloop,}\\
x \, T_{M/e}(x,y) & \text{if $e$ is a coloop,}\\
y \, T_{M/e}(x,y) & \text{if $e$ is a loop.}\end{cases}\]

The Tutte polynomial has been generalized \cite{Oxley-Welsh} by allowing the recursive definition to include coefficients in the case when $e$ is neither a loop nor a coloop.  We say that a function on the class of matroids is a \emph{generalized Tutte--Grothendieck invariant} if it is a function $f$ from the class of matroids to a field $\FF$ such that for all matroids $M$ and $N$, $f(M)=f(N)$ whenever $M\cong N$, and for all $e\in E(M)$,
\[f(M) = \begin{cases}
\sigma f(M-e)+\tau f(M/e) & \text{if is $e$ neither a loop nor a coloop,}\\
f(M(e))f(M-e) & \text{otherwise,}
\end{cases}\]
where $M(e)$ is the matroid consisting of the single element $e$, and $\sigma$ and $\tau$ are nonzero elements of $\FF$.  Let $I$ be the matroid consisting of a single coloop and $L$ be the matroid consisting of a single loop.  We will use the following fact \cite[Theorem 6.2.6]{Oxley-Welsh}:

\begin{theorem}\label{oxleywelshthm}
Let $\sigma$ and $\tau$ be nonzero elements of a field $\FF$.  Then there is a unique function $t'$ from the class of matroids into the polynomial ring $\FF[x,y]$ having the following properties:
\begin{enumerate}[{\rm (i)}]
\item $t'_{I}(x,y)=x$ and $t'_L(x,y)=y$.
\item If $e$ is an element of the matroid $M$ and $e$ is neither a loop nor a coloop, then
\[t'_M(x,y) = \sigma \, t'_{M-e}(x,y)+\tau \, t'_{M/e}(x,y).\]
\item If $e$ is a loop or a coloop of the matroid $M$, then
$t'_M(x,y)=t'_{M(e)}(x,y) \, t'_{M-e}(x,y).$
\end{enumerate}
Furthermore, this function $t'$ is given by
$t'_M(x,y)=\sigma^{|E|-r(E)}\tau^{r(E)} \, T_M (\tfrac{x}{\tau},\tfrac{y}{\sigma} ) \, . $
\end{theorem}

We refer the reader to \cite{Tutte-book},\cite{Welsh-book}, or \cite[Ch.\ 6, Sec.\ 1, 2]{White-book}, for more background on and applications of the (generalized) Tutte polynomial.

Finally, recall that any matrix may be realized as a matroid by taking the ground set to be the list of columns, and the independent sets to be the linearly independent subsets of columns.  If $E$ is the ground set of this matroid, and $y\in \ZZ^{|E|}_q$, then we define the support of $y$ as $\supp(y):=\{e\in E : \, y_e\neq 0\}$. 

\begin{proof}[of Theorem \ref{mainthm}]
Let $\Delta$ be a pure simplicial complex, and $\partial\Delta$ be the boundary matrix associated with $\Delta$.  Let $M$ be the matroid given by the columns of $\partial\Delta$, denoted as the ground set~$E$.

First, we claim that
\[
  T_M(0,1-q)= \left| \left\{ y\in \ker(M)\bmod q : \, \supp(y)=E \right\} \right|
\]
for any prime $q$ that is sufficiently large.  We show this by proving that the function
\[
  g_M(q):=|\{y\in \ker(M)\bmod q : \, \supp(y)=E\}|
\]
is a generalized Tutte--Grothendieck invariant with $\sigma=-1$ and $\tau=1$.  The matrix for the single coloop $I$ is a single (linearly independent) vector, thus $g_I(q)=0$.  The matrix for the single loop $L$ is the zero vector, so $g_L(q)=q-1$.  Thus $g_M(q)$ is well defined if $|E|=1$.

Assume $g$ is well-defined for $|E|<n$, and let $|E|=n$.  Let $e\in E$, and suppose that $e$ is neither a loop nor a coloop.  In the case of a matroid corresponding to a vector configuration (as is the case with the boundary matrix), contraction models quotients: let $V$ be the vector space given by $E$ (the columns of the boundary matrix), and let $\pi:V\rightarrow V/V_e$ be the canonical quotient map, where $V_e$ is the vector space with basis $e$.  Then, the contracted matroid $M/e$ is the matroid associated with the vector configuration $\{\pi(v)\}_{v\in E-\{e\}}$ in the quotient space $V/V_e$.
Let
\begin{align*}
W &:= \{y\in \ker(M) : \, \supp(y)=E\},\\
X &:= \{y\in \ker(M-e) : \, \supp(y)=E-\{e\}\},\mbox{ and}\\
Z &:= \{y\in \ker(M/e) : \, \supp(y)=E-\{e\}\}.
\end{align*}
We see that $W\cap X=\emptyset$.  By linearity of $\pi$, we have that $W\cup X\subset Z$.  Now, suppose we have $y\in Z$.  This is a linear combination of all vectors in $E-\{e\}$ which equals a scalar $\alpha\in\ZZ_q$ times $e$ (that is, a linear combination equivalent to zero in the quotient space).  If $\alpha\neq 0$, then $(y,\alpha)\in W$.  If $\alpha=0$, $y\in X$.  Therefore
\[g_M(q) = g_{M/e}(q)-g_{M-e}(q) \, .\]

If $e$ is a loop (a column of zeros), then $e$ may be assigned any of the $q-1$ nonzero values of $\ZZ_q$.  Thus, for loops
\begin{align*}
g_M(q) &= (q-1) \, g_{M-e}(q)\\
 &= g_L(q) \, g_{M-e}(q) \, .
\end{align*}

Finally, suppose $e$ is a coloop.  Then every maximally independent set of columns contains $e$, and we may, using row operations, rewrite $M$ so that the column corresponding to $e$ has precisely one nonzero element, and the row containing this element also has precisely one nonzero element.  Therefore $y_e=0$ for all $y\in \ker(M)$, and
\begin{align*}
g_M(q) &= 0\cdot g_{M-e}(q)\\
 &= g_I(q) \, g_{M-e}(q) \, .
\end{align*}

Thus $g_M(q)$ is well-defined and a generalized Tutte--Grothendieck invariant.  We showed the cases for $|E|=1$ above, so we see that by Theorem \ref{oxleywelshthm}
\[g_M(q) = t'_M(0,q-1) = (-1)^{|E|-r(E)}T_M(0,1-q)\]
and so
\[T_M(0,1-q)=\left|\left\{y\in \ker(M)\bmod q  : \,  \supp(y)=E\right\}\right| .\]

It follows that the number of nowhere-zero $\ZZ_q$-flows on a simplicial complex $\Delta$ is equal to $T_{\partial\Delta}(0,1-q)$ % (the Tutte polynomial of the boundary matrix of $\Delta$, evaluated at $x=0$, $y=1-q)$), 
and hence is a polynomial in $q$.  Using the definition of the Tutte polynomial, we see that the degree of this polynomial, in terms of the matroid, is $|E|-r(M)$.  From linear algebra, we know that
\[
  |E|=\dim(\ran(M))+\dim(\ker(M))=r(M)+\dim(\ker(M)) \, ,
\]
so $|E|-r(M)=\dim(\ker(M))$.  But $\dim(\ker(M))$ is just the dimension of the top rational homology of $\Delta$, $\dim_{\QQ}(\widetilde{H}_d(\Delta;\QQ))$.  By definition, this is $\beta_d(\Delta)$, where $d=\dim(\Delta)$.
\end{proof}

%\begin{remark}
\noindent  
\emph{Remark.} 
We require that $q$ be prime as otherwise $V$ and $V_e$ would be modules, rather than vector spaces.  Requiring that $q$ be sufficiently large ensures that the matrices reduce correctly over $\FF_q$; that is, we require that the linear (in)dependencies are the same over $\QQ$ as they are over $\FF_q$.  For a simplicial complex of dimension $d$, a sufficient bound would be $(d+1)^{\frac{d+1}{2}}$, though this is not necessarily tight.
%\end{remark}

%%%%%%%%%%%%%%%%%

\subsection{Ehrhart quasipolynomials and proof of Theorem \ref{mainthm2}}

Let $\P \subset \RR^d$ be a \emph{rational polytope}, i.e., the convex hull of finitely many points in $\QQ^d$.
Then Ehrhart's theorem \cite{ehrhartpolynomial,ccd} says that the lattice-point counting function
\[
  \ehr_\P(t) := \# \left( t \P \cap \ZZ^d \right) 
\]
is a \emph{quasipolynomial} in the integer variable $t$, i.e., there exist polynomials $p_0(t), p_1(t), \dots, p_{ k-1 }(t)$ such that
\[
    \ehr_\P(t) = p_j(t) \qquad \text{ if } \qquad t \equiv j \bmod k \, .
\]
The minimal such $k$ is the \emph{period} of $\ehr_\P(t)$ and the polynomials $p_0(t), p_1(t), \dots, p_{ k-1 }(t)$ are its \emph{constituents}.
Since rational polytopes are precisely sets of the form $\left\{ x \in \RR^d : \, Ax \le b \right\}$ for some integral matrix $A$ and integral vector $b$, it is a short step to realize that the flow counting function $\phi_\Delta(q)$ is a quasipolynomial in $q$: it counts the integer lattice points $x = \left( x_1, x_2, \dots, x_d \right) $ (where $d$ is the number of facets of $\Delta$) that satisfy
\[
  0 < x_j < q
  \qquad \text{ and } \qquad
  \partial\Delta (x) = mq \ \text{ for some } \ m \in \ZZ \, .
\]
(Note that we only need to consider finitely many $m$.)
So $\phi_\Delta(q)$ is a sum of Ehrhart quasipolynomials and thus a quasipolynomial in $q$
We remark that a similar setup was used in \cite{breuersanyal} to study flow polynomials of graphs from an Ehrhartian perspective.

\begin{proof}[of Theorem \ref{mainthm2}]
We just showed that $\phi_\Delta(q)$ is a quasipolynomial, say of period $p$.
By Dirichlet's theorem, there exist infinitely many primes of the form $j + kp$ for $\gcd(j,p) = 1$ and $k \in \ZZ_{ \ge 0 }$, and thus $\phi_\Delta(j+kp)$ agrees with the polynomial found in Theorem \ref{mainthm} for those $j$ with $\gcd(j,p) = 1$.
\end{proof}

%%%%%%%%%%%%%%%%%

\subsection{Quasipolynomiality of the Flow Function}
Given that the number of nowhere-zero $\ZZ_q$-flows on a graph is a polynomial in $q$, it is natural to ask whether the same is true for all simplicial complexes; however, this is not always the case.  Consider the case of a Klein bottle, $K$.  We have the following top homologies for $K$:
\[H_2(K,\ZZ_q) = \begin{cases}
0 & \mbox{if }q\mbox{ is odd, and}\\
\ZZ_2 & \mbox{if }q\mbox{ is even.}
\end{cases}\]
Therefore, the number of nowhere-zero $\ZZ_q$-flows on $K$ is given by a quasipolynomial of period $2$:
\[\phi_K(q) = \begin{cases}
0 & \mbox{if }q\mbox{ is odd, and}\\
1 & \mbox{if }q\mbox{ is even.}
\end{cases}\]

%%%%%%%%%%%%%%%%%

\subsection{Calculating a Flow Polynomial: An Example}
Consider the triangular bipyramid $\Delta$ with vertex set $\{0,1,2,3,4\}$ and facets $\{012,013,023,123,124,134,$ $234\}$.  Then $\partial\Delta$ is
\[\begin{array}{c|ccccccc}
 & 012 & 013 & 023 & 123 & 124 & 134 & 234\\
\hline
01 & 1 & 1 & 0 & 0 & 0 & 0 & 0\\
02 & -1& 0 & 1 & 0 & 0 & 0 & 0\\
03 & 0 & -1& -1& 0 & 0 & 0 & 0\\
12 & 1 & 0 & 0 & 1 & 1 & 0 & 0\\
13 & 0 & 1 & 0 & -1& 0 & 1 & 0\\
23 & 0 & 0 & 1 & 1 & 0 & 0 & 1\\
14 & 0 & 0 & 0 & 0 & -1& -1& 0\\
24 & 0 & 0 & 0 & 0 & 1 & 0 & -1\\
34 & 0 & 0 & 0 & 0 & 0 & 1 & 1\\
\end{array}.\]
The kernel is generated by
\[\{(1,-1,1,-1,0,0,0)^T,(0,0,0,1,-1,1,-1)^T\}.\]
Thus, any nowhere-zero $\ZZ_k$-flow will have the form $(a,-a,a,-a+b,-b,b,-b)^T$ where $a,b\in\ZZ_k$ and $a\neq b$.  We have $k-1$ non-zero choices for $a$, and $k-2$ non-zero choices for $b$.  We see that $\phi_{\Delta}(k)=(k-1)(k-2)$.  For primes $k>1$ (1 being the maximum over all absolute values of all subdeterminants of $\partial\Delta$), we may also compute $\phi_{\Delta}(k)=T_{\partial\Delta}(0,1-k)$ as
\begin{align*}
T_{\partial\Delta}(0,1-k) &= \sum_{S\subseteq E}(-1)^{r(M)-r(S)}(-k)^{|S|-r(S)}\\
 &= k^2-3k+2\\
 &= (k-1)(k-2).
\end{align*}

%%%%%%%%%%%%%%%%%

\section{Flows on Triangulations of Manifolds}\label{manifoldsection}

\begin{proof}[of Proposition \ref{manifoldprop}]
Consider a pure simplicial complex that is the triangulation of a connected manifold.

First, if a manifold has boundary, then there exists at least one ridge that belongs to only one facet.  Since this corresponds to a row in its boundary matrix with precisely one nonzero entry, any vector in the kernel must have a zero in the coordinate corresponding to the facet containing this ridge.  Therefore, manifolds with boundary do not admit nowhere-zero flows.

If a triangulated manifold $M$ is without boundary, then every ridge belongs to precisely two facets.  This corresponds to every row of $\partial M$ having exactly two nonzero entries.  Therefore, since our manifold is connected, in a valid flow the assigned value of any facet is equal or opposite mod $q$ to the assigned value of any other facet.  It follows that every flow that is somewhere zero must in fact be trivial.

It is known (see, for instance \cite[Chapter 3.3, Theorem 3.26]{Hatcher-book}) that the top homology of a closed, connected, and $\ZZ$-orientable manifold of dimension $n$ is
\[\widetilde{H}_n(M,\ZZ) \cong\ZZ.\]
In terms of boundary matrices, this means that the kernel of the boundary matrix has rank one, since there are no simplices of dimension higher than $n$.  By our comment above, we see that this element (a linear combination of the facets), must be nowhere-zero, and all entries are $\pm a\in\ZZ$.  It is then easy to see that triangulations of closed, connected, orientable manifolds have precisely $q-1$ nowhere-zero $\ZZ_q$ flows.

If $M$ is non-orientable, connected, and of dimension $n$, then we know (again, see \cite{Hatcher-book}) that the top homology is
\[\widetilde{H}_n(M,\ZZ)=0,\]
so there are no nontrivial flows on triangulations of such manifolds, except when $q=2$.  As every manifold is orientable over $\ZZ_2$, and since every row of the boundary matrix has precisely two nonzero entries, the all-ones vector is a nowhere-zero $\ZZ_2$-flow on $M$ (in fact, the only one).
\end{proof}

%%%%%%%%%%%%%%%%%

\section{Open Questions}
\label{openquestion}
We have shown above that not all simplicial complexes have flow functions that are polynomials.  However, it is still natural to ask for which complexes $\phi_{\Delta}(q)$ \emph{is} a polynomial.  \commentout{Early experiments indicate that simplicial complexes $\Delta$ that are acyclic in positive codimension (that is, $\dim_{\mathbb{Q}}(H_i(\Delta,\mathbb{Q}))=0$ for $i<\dim(\Delta)$), which include classes of complexes such as simplicial spheres, shifted, matroid, and Ferrers complexes, have polynomial flow functions.}  Early experiments indicate this should include simplicial complexes that are convex-ear decomposable (see \cite{Chari}), which include complexes such as matroid, coloring, and a large subclass of unipolar complexes.  We intend to use the structure of these complexes to further specify their specific flow functions.

Many graph polynomials satisfy \emph{combinatorial reciprocity} theorems, i.e., they have an (a priori nonobvious) interpretation when evaluated at negative integers. The classical example is Stanley's theorem connecting the chromatic polynomial of a graph to acyclic orientations \cite{Stanley-recip}; the reciprocity theorem for flow polynomials is much younger and was found by Breuer and Sanyal \cite{breuersanyal}, starting with a geometric setup not unlike that of our proof of Theorem \ref{mainthm2}. We hope to lift their methods into the world of simplicial complexes in the near future. 

%%%%%%%%%%%%%%%%%

\acknowledgements
\label{sec:ack}
The authors would like to thank Eric Babson, Andrew Berget, Felix Breuer, Jesus De Loera, and Steven Klee for suggestions and discussions.

\bibliographystyle{amsalpha}
\nocite{Athanasiadis}
\nocite{Beck-book}
\nocite{Stanley-book}
\nocite{Oxley-book}
\nocite{Duval1}
\nocite{ReinerLec}
\nocite{Stanley-EC1}
\nocite{Kochol}
\nocite{Tutte-book}
\nocite{zhang1997integer}
\bibliography{simpflows}
\label{sec:biblio}

\end{document}